\documentclass{amsart}
\usepackage{amssymb}
\usepackage{amsthm}
\usepackage{mathtools}
\usepackage{environ}
\usepackage{enumerate}
\usepackage{microtype}
\usepackage{newtxmath}
\usepackage{enumitem}
\usepackage[utf8]{inputenc}

\theoremstyle{plain}
\newtheorem{theorem}{Theorem}[section]
\newtheorem{lemma}[theorem]{Lemma}
\newtheorem{proposition}[theorem]{Proposition}
\newtheorem{corollary}[theorem]{Corollary}

\theoremstyle{definition}
\newtheorem{definition}[theorem]{Definition}
\newtheorem{example}[theorem]{Example}

\theoremstyle{remark}
\newtheorem*{remark}{Remark}

\NewEnviron{s}{%
\begin{equation*}\begin{split}
  \BODY
\end{split}\end{equation*}
}

\newcommand{\bd}{\partial}

\newcommand{\C}{\mathbb{C}}
\newcommand{\R}{\mathbb{R}}
\newcommand{\N}{\mathbb{N}}
\newcommand{\psh}{\mathcal{PSH}}
\newcommand{\usc}{\mathcal{USC}}
\newcommand{\lsc}{\mathcal{LSC}}
\newcommand{\suchthat}{\mathrel{;}}
\DeclarePairedDelimiter\abs{\lvert}{\rvert}
\DeclarePairedDelimiter\set{\{}{\}}

\title[Continuity of envelopes of unbounded plurisubharmonic functions]{Continuity of envelopes of \\ unbounded plurisubharmonic functions}
\author{Mårten Nilsson}
\address{Center for Mathematical Sciences\\
  Lund University\\
  Box 118, SE-221 00 Lund, Sweden}

\email{marten.nilsson@math.lth.se}


\begin{document}

\maketitle
\thispagestyle{empty}
\begin{abstract}
On bounded B-regular domains, we study envelopes of plurisubharmonic functions bounded from above by a function $\phi$ such that $\phi^*=\phi_*$ on the closure of the domain. For $\phi$ satisfying certain additional criteria limiting its behavior at the singularities, we establish a set where the Perron--Bremermann envelope $P \phi$ is guaranteed to be continuous. This result is a generalization of a classic result in pluripotential theory due to J. B. Walsh. As an application, we show that the complex Monge--Ampère equation
\[
(dd^cu)^n = \mu
\]
being uniquely solvable for continuous boundary data implies that it is also uniquely solvable for a class of boundary values unbounded from above.
\end{abstract}
\section{Introduction}

We say that a plurisubharmonic function $u$ is \textit{quasibounded} if there exists an increasing sequence $\{u_n\}$ of upper bounded, plurisubharmonic functions such that $u_n \nearrow u$. If the convergence only holds outside a pluripolar set, we say that $u$ is \textit{quasibounded quasi-everywhere}. It was established in~\cite{nilsson-wikstrom} that a sufficient condition for the latter is that $u$ has a \textit{tame} majorant, i.e. a non-negative function $f \geq u$ for which there exists a function $\psi\colon [0, +\infty] \to [0,+\infty]$ such that $\psi(+\infty) = +\infty$,
    \begin{equation*}
        \lim_{t\to+\infty} \frac{\psi(t)}{t} = +\infty,
    \end{equation*}
    and $\psi \circ f$ admits a non-trivial plurisuperharmonic majorant. One may equivalently define tame functions in a less straight-forward way by first considering the set
\[
    \mathcal{M} := \big\{ f\colon \Omega \to [0,+\infty] \suchthat \exists u \in -\psh(\Omega), f \le u \big\},
\]
i.e.\ the set of non-negative functions on a bounded domain $\Omega \subset \C^n$ admitting a plurisuperharmonic majorant. On $\mathcal{M}$, one may define a family of operators $S_\lambda: \mathcal{M} \to -\psh(\Omega)$ by
\[
    (S_\lambda f)(z) := \Big(\inf \big\{ v(z) \suchthat v \in -\psh(\Omega), v \ge (f-\lambda)^+\big\}\Big)_*,
\]
and it was proven in \cite{nilsson-wikstrom} that  $S(f):=(\lim_{\lambda \rightarrow \infty} S_\lambda (f))_*\equiv 0$ if and only if $f$ is tame in the above sense. The analogous characterization in the harmonic case was established by Arsove and Leutwiler~\cite{arsove}. In Section 2, we establish a third characterization of tameness, and use this to establish a continuity set for envelopes of the form
    \[
        P\phi(z) = \sup \big\{ u(z) \suchthat u \in \psh(\Omega), u^* \le \phi \big\},
    \]
where $\phi$ is a lower bounded function with $\phi^* = \phi_*$ on $\bar\Omega$, such that the positive part of $\phi$ is tame. Recall that
\[
\phi^*(z) := \limsup_{\Omega \ni w\rightarrow z} \phi(w) \qquad \phi_*(z) := \liminf_{\Omega \ni w\rightarrow z} \phi(w)
\]
for all $z \in \bar \Omega$. In particular, $u^*$ is upper semicontinuous on $\bar \Omega$ for any upper bounded plurisubharmonic function $u$ defined on $\Omega$. As shown in ~\cite{nilsson-wikstrom}, such envelopes may be used to show that the Dirichlet problem
    \[
    \begin{cases}
    u \in \psh(\Omega)\cap L^\infty_\text{loc} (\Omega)  \\
    (dd^cu)^n=0 \\
    u^* \leq \phi \\
        \lim_{\Omega \ni \zeta \rightarrow z_0 \in \bd \Omega } u(\zeta) = \phi(z_0)
    \end{cases}
    \]
is uniquely solvable when $\Omega$ is B-regular and $\phi$ is a tame harmonic function satisfying $\phi^* = \phi_*$. Using similar techniques, we show in Section 3 that a large class of corresponding inhomogeneous Monge--Ampère equations also have unique solutions, under an additional reasonable assumption on $\phi$.   

As harmonicity is preserved when multiplying by $-1$, the original theory of Arsove and Leutwiler may almost verbatim be adapted to negative functions. The corresponding notion of tameness in pluripotential theory, equivalent to the definition above up to sign, has recently been used by Rashkovskii ~\cite{rashkovskii} in the study of singularities of negative plurisubharmonic functions, motivated by the problem of determining when two plurisubharmonic functions may be connected by a geodesic. In the final section of the paper, we therefore investigate the envelope $P\phi$ when $\phi$ is allowed to have negative poles as well. Using the machinery of Jensen measures and a generalization of Edwards' theorem, it is again possible to establish a set where the envelope is continuous. More explicitly, we are able to show that $P\phi$ is continuous on the set
\[
\{z \in \bar \Omega \suchthat v_*(z) \neq -\infty\} \cap \{z \in \bar \Omega \suchthat w^*(z) \neq +\infty\},
\]
where $v, -w \in \psh(\Omega) \cap \usc (\bar \Omega)$ and the combined singularities of $v, w$ in a precise sense surpass those of $\phi$. 

The author would like to thank Frank Wikström and Filippo Bracci for helpful discussions. 
 
\section{Envelopes bounded from below}
We begin with the following definition, which we will refer to throughout the paper.
\begin{definition}\label{majmin}
For a given pair $f:\bar \Omega \rightarrow [-\infty, +\infty]$, $v\in -\psh(\Omega) \cap \lsc (\bar \Omega)$, we say that $v$ is a \textit{strong majorant} to $f$ if
    \begin{s}
      f(z_0)=+\infty \implies v(z_0)=+\infty, \\
      \frac{v(z)}{f(z)} \rightarrow \infty \text{ as }f(z) \rightarrow \infty.
    \end{s}
We say that a function $w$ is a \textit{strong minorant} to $f$ if $-w$ is a strong majorant to $-f$.
\end{definition}
It turns out that for positive functions, having a non-trivial strong majorant is equivalent to being tame.
\begin{lemma}\label{tam-majorant}
Let $f:\bar \Omega \rightarrow [0, +\infty]$. Then the following are equivalent: 
\begin{enumerate}[topsep=10pt,itemsep=1ex,partopsep=1ex,parsep=2ex]
    \item $S(f) = 0$.
    \item There exists a function $\psi\colon [0, +\infty] \to [0,+\infty]$ such that $\psi(+\infty) = +\infty$,
    \begin{equation*}
        \lim_{t\to+\infty} \frac{\psi(t)}{t} = +\infty,
    \end{equation*}
    and $\psi \circ f$ admits a non-trivial plurisuperharmonic majorant.
    \item $f$ has a non-trivial strong majorant.
\end{enumerate}
\end{lemma}
\begin{proof}
(3) $\implies$ (1): Let $v$ be a strong majorant to $f$. Then $v_\varepsilon := \varepsilon v$ are strong majorants to $f$ as well, for all $\varepsilon > 0$. Define $C=\{z \suchthat v_\varepsilon(z)<f(z)\}$, and note that $f$ is bounded from above by some constant $k>0$ on $C$, since otherwise we could find a sequence $z_n \in C$ such that $f(z_n) \rightarrow \infty$, reaching a contradiction. It follows that $f \leq v_\varepsilon + k$, and thus
\begin{s}
S(f) &\leq \varepsilon S(v).
\end{s}
Letting $\varepsilon \rightarrow 0$, we get $S(f)=0$.

 (1) $\iff$ (2) $\implies$ (3): Let $v$ be a plurisuperharmonic majorant to $\psi \circ f$, and let $k>0$ be such that $\psi(t)+k > t$. Clearly $v + k$ is a strong majorant to $f$, and as (1) $\iff$ (2) was proven in \cite{nilsson-wikstrom}, the lemma follows.
\end{proof}

Working with condition (3) instead of condition (2), the proof of Theorem 4.4 in \cite{nilsson-wikstrom} may be adapted to establish the following general statement.
\begin{proposition}\label{nedat-begransad}
Let $\Omega$ be a bounded B-regular domain, and let $\phi$ be a function bounded from below such that $\phi^*=\phi_*$ on $\bar \Omega$. Then the envelope
 \[
 P\phi(z) = \sup \set{ u(z) \suchthat u \in \psh(\Omega), u^* \le \phi }
 \]
is continuous on $\{z \in \bar \Omega \suchthat v^*(z) \neq +\infty\}$, where $v$ is a strong majorant to $\phi$.
\end{proposition}
\begin{proof}
Avoiding a vacuous truth, we may assume that $v$ is a non-trivial majorant and that $\{z \in \bar \Omega \suchthat v^*(z) \neq +\infty\}$ is non-empty. Since $\phi$ is bounded from below, we may further assume that $v\geq \phi > 0$. Lemma~\ref{tam-majorant} then implies that $S(\phi)=0$, and hence every $u \in \psh(\Omega)$ satisfying $u^* \leq \phi$ is quasibounded quasi-everywhere. On the other hand, the sequence of harmonic functions $\{h_n\}$ defined by $h_n|_{\bd \Omega} = \min\{\phi,n\}$ is also majorized by $v$, therefore converging to a harmonic function $h$ by Harnack's theorem. Clearly $h\geq u$, and it follows from the Brelot--Cartan theorem that $P\phi$ is upper semicontinuous and plurisubharmonic on $\Omega$. 

In order to establish continuity at all points $z$ where $v^*(z) \neq +\infty$, we will now show that
 \[
 P\phi(z) = \sup \set{ u(z) \suchthat u \in \psh(\Omega)\cap C(\bar \Omega), u \le \phi }
 \]
for all such $z$. Continuity will then follow as $\{z \in \bar \Omega \suchthat v^*(z) \neq +\infty\}$ is open, and the pointwise supremum of a family of lower semicontinuous functions on a open set is lower semicontinuous.  By Wikström~\cite[Corollary 4.3]{wikstrom},
\begin{align*}
&\sup \set{ u(z) \suchthat u \in \psh(\Omega)\cap C(\bar \Omega), u \le \phi } \\
&\qquad= \sup \set{ u(z) \suchthat u \in \psh(\Omega)\cap \usc(\bar \Omega), u \le \phi }
\end{align*}
and thus it is enough to find a sequence of upper bounded plurisubharmonic functions approximating $P\phi$ from below. As in the proof of Lemma~\ref{tam-majorant}, pick $k>0$ such that $\varepsilon v+k>\phi \geq P\phi$. Then $u_\varepsilon := P\phi - \varepsilon v$ yields such a sequence as $\varepsilon \rightarrow 0$, since
\begin{s}
  P\phi - \varepsilon v= P\phi - \varepsilon S_0(v) &= P\phi -  S_{k}(\varepsilon v+k)  \\ 
  &\leq \max\{P\phi,0\} -  S_{k}(\max\{P\phi,0\}) \leq k <\infty. \qedhere
\end{s}
\end{proof}
\begin{remark}
Using different methods, it is possible to prove a similar statement without the requirement that $\phi$ is bounded from below. We will return to this question in Section~\ref{sec:obunden}. 
\end{remark}
\section{A Dirichlet problem with unbounded boundary data}
We now turn our attention to the inhomogeneous Dirichlet problem for which the boundary data corresponds to a function on $\bar \Omega$ of the same flavor as in Proposition~\ref{nedat-begransad}. We will restrict our analysis to measures for which the complex Monge--Ampère equation is solvable when considering continuous boundary data. 
\begin{definition}
We say that a measure $\mu$ is \textit{compliant} if
    \[
    \begin{cases}
    u \in \psh(\Omega)\cap L^\infty(\Omega)  \\
    (dd^cu)^n=\mu \\
        \lim_{\Omega \ni \zeta \rightarrow z_0 \in \bd \Omega } u(\zeta) = \phi(z_0)
    \end{cases}
    \]
has a unique solution for all $\phi \in C(\bd \Omega)$. If every such solution is continuous, we say that $\mu$ is \textit{continuously compliant}.
\end{definition}
\begin{remark}
 Ko\l{}odziej \cite{kolod1} proved in particular that the measure $|f|\,dV$ is continuously compliant if $f \in L^p (\Omega)$, where $p>1$ and $\Omega$ is strictly pseudoconvex. Mere compliance of a measure may be shown by applying Ko\l{}odziej's subsolution theorem \cite{kolod2}. Note that the existence of continuously compliant measures implies that $\Omega$ is B-regular. 
\end{remark}

It should be mentioned that the proof of the homogeneous case addressed in \cite{nilsson-wikstrom} only needs minor modifications to encompass boundary values corresponding to a tame \textit{harmonic} function $\phi$, bounded from below such that $\phi^*=\phi_*$ on $\bar \Omega$. Proposition~\ref{nedat-begransad} then implies that the solution is continuous off the singularity set of any strong majorant of $\phi$. In the inhomogeneous case however, we will make an additional technical assumption that there exists a plurisubharmonic function $\xi$ such that $-\xi \geq \phi$ and $\xi$ lies in the Cegrell class $\mathcal{E}(\Omega)$. This class may be defined as all negative plurisubharmonic functions $u$ such that for each $z_0 \in \Omega$, there exists a neighborhood $U_{z_0}$ and decreasing sequence of bounded functions $h_j$ such that $\lim_{z \rightarrow \zeta_0 \in \bd \Omega}h_j(z)=0,$ 
\[
 \sup_j \int (dd^c h_j)^n < \infty,
\]
and $h_j \searrow u$ on $U_{z_0}$. The class $\mathcal{E}(\Omega)$ is in a certain sense the largest domain for which the Monge--Ampère operator is well-behaved, but requires hyperconvexity of the domain $\Omega$ not to be empty. See Cegrell \cite{cegrell} for details.

We will make heavy use of the following result due to Demailly~\cite{demailly}.
\begin{lemma}\label{demailly}
Let $\Omega$ be a domain in $\C^n$, $u,v \in \psh(\Omega)\cap L_{loc}^\infty (\Omega)$. Then
\[
(dd^c \max\{u,v\})^n \geq \chi_{u\geq v}(dd^c u)^n + \chi_{u<v}(dd^c v)^n.
\]
\end{lemma}

We are now ready to prove the main theorem of this section.
\begin{theorem}\label{thm:dirichlet-problem}
    Let $\Omega \subset \mathbb{C}^n$ be a bounded hyperconvex domain, $\mu$ be a compliant measure, and let $\phi$ be a harmonic function bounded from below such that $\phi^* = \phi_*$ on $\bar\Omega$, $\phi$ has a non-trival strong majorant $\psi$ and there exists $\xi \in \mathcal{E}(\Omega)$ such that $-\xi \geq \phi$. Then the Dirichlet problem
    \[
    \begin{cases}
    u \in \psh(\Omega)\cap L^\infty_\text{loc} (\Omega)  \\
    (dd^cu)^n=\mu \\
    u^* \leq \phi \\
        \lim_{\Omega \ni \zeta \rightarrow z_0 \in \bd \Omega } u(\zeta) = \phi(z_0)
    \end{cases}
    \]
    has a unique solution. If $\mu$ is continuously compliant, then the solution is continuous outside the set $\{z \suchthat \psi^*(z) = + \infty\}$.
\end{theorem}

\begin{proof}
     A candidate for a solution is the envelope
    \[
    U = \sup \big \{u \suchthat u \in \psh (\Omega)\cap L^\infty_\text{loc} (\Omega) , (dd^cu)^n \geq \mu, u\leq \phi \big \}.
    \]
    Note that the defining family for $U$ is non-empty, since it contains functions $v_k$ satisfying 
    \[
    \begin{cases}
    v_k \in \psh(\Omega)\cap L^\infty(\Omega)  \\
    (dd^c v_k)^n= \mu \\
        \lim_{\Omega \ni \zeta \rightarrow z_0 \in \bd \Omega } v_k(\zeta) = \min\{\phi(z_0), k\}.
    \end{cases}
    \]
    by the compliance of $\mu$. By Choquet's lemma, $U^* = (\sup_k U_k)^*$ for some countable family $\{U_k\}$. Since $\tilde U_k := \max\{U_1,..., U_k\}$ increases to $U^*$ outside a pluripolar set, we conclude that $(dd^cU^*)^n \geq \mu$, and so $U^*=U$. Without loss of generality, assume that $U \geq v_0 > 0$.
    
    By Lemma~\ref{tam-majorant}, $S(U)=0$. Following ~\cite{nilsson-wikstrom}, we set
    \[
    w_k := U - S_k(U),
    \]
    where $w_k$ is bounded from above and $w_k \nearrow U$ quasi-everywhere. Since $-S_k(U) =\max\{-S_k(U),\xi\}$, we may conclude (cf. Cegrell~\cite[Theorem 4.5]{cegrell}) that $-S_k(U)$ is an element of $\mathcal{E}(\Omega)$. It follows from regularization and Minkowski's determinant inequality that $(dd^cw_k)^n \geq \mu$. Now for $m < \min_{z\in \bar \Omega} v_0$ we may write
    \[
    u_k := \max\{v_k, w_k\} =\max\big\{v_k, \max\{w_k,m\}\big\},
    \]
    where
    \[
    (dd^cu_k)^n \geq \chi_{v_k \geq \max\{w_k,m\}} (dd^c v_k)^n +\chi_{\max\{w_k,m\} \geq v_k}\big(dd^c\max\{w_k,m\}\big)^n
    \]
    by Lemma~\ref{demailly}. Letting $m \rightarrow -\infty$, we infer that
    \[
    (dd^cu_k)^n \geq \mu.
    \]
    It follows from the domination principle that $u_k \leq v_k$, and thus $v_k \nearrow U$ quasi-everywhere. We conclude that $(dd^cU)^n = \mu$.
    
     In order to show that the boundary values are attained, fix $z_0 \in \bd \Omega$, and note that
    \begin{align*}
        \min\{\phi(z_0),k\} &= v_k(z_0) =  
        \liminf_{\Omega \ni \zeta \rightarrow z_0} v_k(\zeta) \leq  
        \liminf_{\Omega \ni \zeta \rightarrow z_0} U(\zeta) \\
        &\leq
    \limsup_{\Omega \ni \zeta \rightarrow z_0} U(\zeta) \leq \limsup_{\Omega \ni \zeta \rightarrow z_0} \phi(\zeta) = \phi(z_0).
    \end{align*}
    By letting $k \to \infty$, the above chain of inequalities readily implies that
    \[
    \lim_{\Omega \ni \zeta \rightarrow z_0} U(\zeta) = \phi(z_0).
    \]
    
    The proof of uniqueness is essentially the same as in the proof of Theorem 4.2 in \cite{nilsson-wikstrom}, but we will include it here for completeness. Suppose that $ \tilde u$ is another solution to the Dirichlet problem. Note that $\tilde u \leq U$, as $\tilde u$ lies in the defining family for $U$, with $\tilde u < U$ on some set that is not pluripolar. It follows that there exists an approximant $u_k$ and a constant $C>0$ such that the set  
    \[
    E=\{z \in \Omega : \tilde u + C < u_k(z) \}
    \]
    is non-empty as well, and the boundary values of $\tilde u$ guarantees that $E \Subset \Omega$. Thus we may find an open set $\tilde \Omega$ such that $E \Subset \tilde \Omega \Subset \Omega$ and $\tilde u + C \geq u_k$ on $\bd \tilde \Omega$. Since 
    \[
    (dd^c \tilde u)^n  = \mu \leq (dd^c u_k)^n, 
    \]
    the comparison principle implies that $E$ is empty.
    
    Lastly, assume that $\mu$ is continuously compliant. Since the Monge--Ampère operator is continuous with regards to increasing sequences and $U$ is the unique solution, we may write
    \[
    U = \big(\lim_{k\rightarrow \infty} v_k\big)^*,
    \]
    and it follows from the Brelot--Cartan theorem that $U$ coincides with a lower semicontinuous function outside of a pluripolar set, i.e. continuous outside a pluripolar set $P$. We claim that $P \subset \{z \suchthat \psi(z) = + \infty\}$. As before,  pick $k>0$ such that $\varepsilon \psi+k>\phi \geq U$, and define $\tilde v_k = \max\big\{U + \max\{-\varepsilon \psi,  \xi \}, v_k\big\}$. Then by the proof of Proposition~\ref{nedat-begransad},
    \[
    v_k \leq \tilde v_k \leq \min\{\phi, k\},
    \]
    while Lemma~\ref{demailly} implies that
    \begin{s}
    (dd^c \tilde v_k)^n &\geq  \chi_{\max\{U + \max\{-\varepsilon \psi,\xi\},0\} \geq v_k} (dd^c \max\{U + \max\{-\varepsilon \psi,\xi\},0\})^n \\
    &\qquad + \chi_{\max\{U + \max\{-\varepsilon \psi,\xi\},0\} < v_k} (dd^c v_k)^n \\
    &= \chi_{U + \max\{-\varepsilon \psi,\xi\} \geq v_k} (dd^c U + \max\{-\varepsilon \psi,\xi\})^n \\
    &\qquad + \chi_{\max\{U + \max\{-\varepsilon \psi,\xi\},0\} < v_k} (dd^c v_k)^n \\
    &\geq \mu = (dd^c v_k)^n,
    \end{s}
so by the domination principle, $\tilde v_k=v_k$. It follows that $v_k \nearrow U$ outside $\{z \suchthat \psi(z) = + \infty\}$, and that $U$ is continuous on $\{z \suchthat \psi^*(z) \neq + \infty\}$.
\end{proof}
\begin{remark}
In the general case when $\phi$ is not necessarily tame, the uniqueness argument above shows that any solution must be larger or equal to $(\sup_k v_k)^*$. 
\end{remark}
We stress that having a strong majorant is only a sufficient condition. When $\mu = 0$, it was shown in \cite{nilsson-wikstrom} that uniqueness is equivalent to the envelope above being quasibounded quasi-everywhere. Indeed, tame plurisubharmonic functions are characterized by the following, possibly stronger property:
\begin{theorem}
 Let $0 \leq u \in \psh(\Omega)$. Then $u$ has a tame majorant if and only if there exists an increasing sequence of upper bounded plurisubharmonic functions $u_n\leq u$ such that $(u_n-u)^*$ is plurisubharmonic and $u_n-u \nearrow 0$ quasi-everywhere. 
\end{theorem}
\begin{proof}
 If $u$ is tame, then such a sequence exists by \cite[Theorem~3.1]{nilsson-wikstrom}. For the other direction, let $K_n \geq 0$ be constants such that $u_n \leq K_n$. Then $(u-u_0)_*+K_0$ is a plurisuperharmonic majorant to $u$, so $u\in \mathcal{M}$. Using the properties of $S$ established in \cite{nilsson-wikstrom},
\[
S(u) =  S(u- K_n )  \leq S(u - u_n) \leq u - u_n  \rightarrow 0
\]
outside a pluripolar set, and it follows that $u$ is tame.
\end{proof}

Assuming that $0 \leq \phi \in -\psh(\Omega)$, one could ask whether $S(\phi)=0$ (i.e. $\phi$ is tame) implies that $-\phi \in \mathcal{E}(\Omega)$, and vice versa. As the following examples show, this is not case.
\begin{example}
    Let $\Omega = \mathbb{B}$ be the unit ball in $\mathbb{C}^2$ and let $\phi(z, w) = (- \log\abs{z})^\alpha$, where $1/2 \leq
    \alpha < 1$. Then $\phi$ is tame by \cite[Corollary~3.2]{nilsson-wikstrom}, but $-\phi \not \in \mathcal{E}(\Omega)$ by Cegrell~\cite[Example~5.9]{cegrell}.
\end{example}
\begin{example}
    Let $\Omega = \mathbb{B}$ be the unit ball in $\mathbb{C}^2$. Define
    \[
        \phi(z,w)=\frac{1-\abs{z}^2}{\abs{1-z}^2}.
    \]
    By \cite{nilsson-wikstrom}, $\phi$ is not quasibounded, and in particular not tame. On the other hand, $-\phi \in \mathcal{E}(\Omega)$ by B\l{}ocki~\cite[Theorem~4.1]{blocki}.
\end{example}

In a recent paper, Rashkovskii~\cite{rashkovskii} established sufficient conditions for $\phi$ to satisfy both these conditions, yielding more examples of tame functions than one may extract from Lemma~\ref{tam-majorant}. Specifically, he introduced the notion of the Green--Poisson residual function $g_{\phi}$ for $\phi \in \psh^-(\Omega)$, defined as
\[
g_{\phi}(z): = \limsup_{x\rightarrow z}
\sup\{v(x) \suchthat v \in \psh^- (\Omega), v \leq \phi + C_v\}.
\]
This construction is very similar to the construction of the operator~$S$, and in fact, for $\phi \in \mathcal{M}\cap -\psh(\Omega)$, 
\[
S(\phi) = -g_{-\phi}.
\]
To see this, note that
\begin{align*}
  (S_\lambda \phi)(z) &= \big(\inf \big\{ v(z) \suchthat v \in -\psh(\Omega), v \ge (\phi-\lambda)^+\big\}\big)_* \\
  &= \big(\inf \big\{ v(z) \suchthat v \in  \mathcal{M}\cap-\psh(\Omega), v \ge (\phi-\lambda)^+\big\}\big)_* \\
  &=-\big(\sup \big\{ v(z) \suchthat v \in \psh^-(\Omega), v \le \min\{-\phi+\lambda,0\} \big\}\big)^* \\
  &=-\big(\sup \big\{ v(z) \suchthat v \in \psh^-(\Omega), v \le -\phi+\lambda \big\}\big)^*.
\end{align*}
Letting $\lambda \rightarrow \infty$, we encompass all negative plurisubharmonic functions $v$ for which there exists a constant $C_v$ with $v \leq -\phi + C_v$.

Using Corollary~6.10 in \cite{rashkovskii} and Lemma ~\ref{tam-majorant}, we immediately obtain the following corollaries. Here, $\mathcal{E}^a(\Omega) \subset \mathcal{E}(\Omega)$ denotes the subset of functions whose Monge--Ampère measure does not charge pluripolar sets. 
\begin{corollary}
Let $\Omega, \Omega'$ be two hyperconvex domains such that $\Omega \Subset \Omega'$. Then every function having a plurisuperharmonic majorant $\phi$ with $-\phi \in \mathcal{E}(\Omega')\cap\mathcal{E}^a(\Omega)$ is tame.
\end{corollary}
\begin{corollary}\label{tam-minorant} Let $\Omega$ be a bounded domain, and let $\phi \in \psh^-(\Omega)$. Then $g_\phi = 0$ if and only if $\phi$ has a non-trivial strong minorant.
\end{corollary}
\section{Unbounded envelopes}\label{sec:obunden}
As Lemma ~\ref{tam-majorant} and Corollary ~\ref{tam-minorant} shows, the language of strong majorants and minorants is appropriate when limiting the behavior of positive and negative singularities all at once, for a function $\phi:\bar \Omega \rightarrow [-\infty, +\infty]$. In this section, we expand our study to functions of this sort. We begin by establishing a preliminary result of independent interest.

Let $X$ be a compact metric space, and let $\mathcal{F}$ be a cone of upper bounded, upper semicontinuous functions on $X$, not necessarily containing zero. For each $x \in X$, we associate a set of positive measures
\[
M_x^{\mathcal{F}}= \big\{\mu\suchthat u(x)\leq \int u \,d\mu \text{ for all } u \in \mathcal{F} \big\}.
\]
Further, for a measurable function $g$ on $X$, we define
\begin{s}
S_x(g) &= \begin{cases}
\sup\{u(x) \suchthat u\in \mathcal{F}, u \leq g\}, & \hfill \text{if }\exists u\in \mathcal{F} \suchthat u \leq g, \\
- \infty, &\hfill \text{otherwise,}
\end{cases}
\\
 I_x(g)&=\inf\big\{\int g \, d\mu \suchthat \mu \in M_x^{\mathcal{F}}\big\}.
\end{s}
It is clear that $S_x(g)\leq I_x(g)$, since for all $u$ in the defining envelope of $S_x(g)$, and all $\mu$ in the defining envelope of $I_x(g)$,  we have
\[
u(x) \leq \int u \, d\mu \leq \int g \, d\mu.
\]
Edwards' theorem \cite{edwards} states the following:

\begin{theorem} If $\mathcal{F}$ contains all constants and $g$ is bounded and lower semicontinuous, then $Sg=Ig$. 
\end{theorem}

It is possible to modify Edwards' proof to prove a more general statement, valid for cones not containing constants and for $g$ not bounded from above. To this end, we consider the set of measures
\[
C_x^{\mathcal{F}}:= \big\{\mu\suchthat \mu(X) \leq -S_x(-\vmathbb{1}) \big \} \cap M_x^{\mathcal{F}},
\]
where $-\vmathbb{1}(x)\equiv -1$. It follows from the Banach–Alaoglu theorem and the fact that upper semicontinuous functions may be approximated from above by continuous functions that $C_x^{\mathcal{F}}$ is weak$^*$-compact for all $x\in X$ such that $S_x(-\vmathbb{1})>-\infty$. Letting these sets carry the compactness argument of the original proof, we obtain the following reformulation of Edwards' theorem.

\begin{theorem}\label{edwards} Let $g$ be lower semicontinuous and bounded from below. For all $x\in X$ such that $S_x(-\vmathbb{1})>-\infty,$ we have $S_x(g)=I_x(g)$. It is enough to take the infinum over measures in $C_x^{\mathcal{F}}.$
\end{theorem}


\begin{proof}
\textit{Step 1. } We first consider the case when $g \in C(X)$. Since $g$ is bounded from below by say $-K$, we have
\[
\infty > g(x) \geq S_x(g) \geq K \cdot S_x(-\vmathbb{1}) > -\infty.
\]
In other words, the map $S_x$ on measurable functions defined by $g \mapsto S_x(g)$ maps $C(X)$ into $\R$. Furthermore, $-S_x$ is sublinear:
\begin{s}
i)& -S_x(\alpha g) = -\alpha S_x(g),\ \alpha \geq 0, \\
ii)& -S_x(g_1 + g_2) \leq -S_x(g_1) - S_x(g_2).
\end{s}

On the subspace of $C(X)$ generated by $g$, we now define the linear function $-H_x(k g):=-kS_x(g)$, where $k \in \R$. Now note that for $k\geq 0$, 
\[
-S_x(kg)=-H_x(k),
\]
and using $ii)$, we have
\[
 0 = -S_x(kg-kg) \leq -S_x(kg)-S_x(-kg)=H_x(-kg)-S_x(-kg),
 \]
implying $-H_x(-kg)\leq-S_x(-kg)$. It follows that $-H_x \leq -S_x$ on the whole span generated by $g$, and by the Hahn-Banach theorem, we may extend $-H_x$ to a linear map $-\tilde H_x$ on $C(X)$, dominated by $-S_x$. 

For $h \in C(X), h\geq 0,$ we have
\[
\tilde H_x(h) \geq S_x(h) > \varepsilon S_x(-\vmathbb{1}) > -\infty,
\]
and letting $\varepsilon \rightarrow 0$, $\tilde H_x(h) \geq 0$. The Riesz representation theorem yields a measure representation of $\tilde H_x$, which we denote by $\mu_g$. Clearly $\mu_g \in C_x^{\mathcal{F}} \subset M_x^{\mathcal{F}}$, so $S_x(g)$ and $I_x(g)$ do indeed coincide. 

\textit{Step 2.} If $g$ is merely lower semicontinuous and bounded from below, note that $S_x(g)=I_x(g)$ if $S_x(g)=\infty$, so assume $S_x(g)<\infty$. Now, let $g_i \in C(X)$ be an increasing sequence such that $g_i \nearrow g$. Further, associate to each $g_i$ a measure $\mu_{g_i}$ as above and extract a subsequence  $\{\mu_{g_{i_n}}\}$ converging to $\mu_g$ in the weak$^*$-topology. Fixing $k\in \N$, we have
\begin{s}
  S_x(g) &\geq \lim_{n\rightarrow \infty} Sg_{i_n}(x) = \lim_{n\rightarrow \infty} Ig_{i_n}(x) =\lim_{n \rightarrow \infty}\int g_{i_n} \, d\mu_{g_{i_n}} \\
  &\geq \lim_{n\rightarrow \infty}\int g_k \, d\mu_{g_{i_n}} = \int g_k \, d\mu_{g}. 
\end{s}
Now take a constant $K$ such that $g_k+K>0$ for all $k$, and write
\[
\int g_k \, d\mu_{g}= \int (g_k+K) \, d\mu_{g} - \int K \, d\mu_{g}
\]
and apply the monotone convergence theorem to the sequence $\{g_k+K\}$ in order to conclude $S_x(g) \geq I_x(g)$, and thus $S_x(g) = I_x(g)$. 
\end{proof}
\begin{corollary}
If $\mathcal{F}$ contains negative constants and $g$ is lower semicontinuous, then $Sg=Ig$.
\end{corollary}
\textit{Remark.} The first step in the proof above does not require that the cone consists of upper semicontinuous functions. Also note that $S_x(g)>-\infty$ for all $g \in C(X)$ if and only if $S_x(-\vmathbb{1})>-\infty$, so the assumptions of the theorem cannot be weakened whilst still using the proof idea of Edwards (viewing $S_x$ as a supralinear operator on $C(X)$).

Now, Theorem ~\ref{edwards} does not say anything 
if $S_x(-\vmathbb{1})=-\infty$ for all $x\in X$. In certain situations, this may be mended by embedding the cone into a larger cone on which Theorem ~\ref{edwards} may be applied. 
\begin{proposition}
Let $g$ be lower semicontinuous, and suppose that $\mathcal{F}$ is a cone such that for all $K>0$,
\[
u \in \mathcal{F} \implies \max\{u-K, 0\} \in \mathcal{F}.
\]
If $g\geq 0$, then $S(g)=I(g)$. 
\end{proposition}
\begin{proof}
Define the cone $\mathcal{F}^*$, generated by $\mathcal{F}$ and $-\vmathbb{1}$, and denote  the corresponding operators by $S^*, I^*$. This way, we get more functions and fewer measures, i.e. $\mathcal{F} \subset \mathcal{F}^*$ and $M_x^{\mathcal{F}^*} \subset M_x^{\mathcal{F}}$, and thus we have
\[
S_x(g) \leq I_x(g) \leq  I^*_x(g)=S^*_x(g),
\]
since $\mathcal{F}^*$ satisfies the assumptions of Theorem ~\ref{edwards} at all points $x \in X$. Hence it is enough to show that $S_x(g)=S^*_x(g)$. However, any element in the defining envelope for $S^*g$ may be written $u-K\leq g$, where $u \in \mathcal{F}$ and $K> 0$. As $u-K \leq \max\{u-K,0 \} \in \mathcal{F}$, indeed $S_x(g)=S^*_x(g)$ for $g\geq 0$.
\end{proof}
\begin{proposition}
Let $g\leq 0$ be lower semicontinuous, and suppose that $\mathcal{F}$ is a cone of non-positive functions such that for all $K>0$,
\[
u \in \mathcal{F} \implies \max\{u, -K\} \in \mathcal{F}.
\]
If $g$ has a minorant in $\mathcal{F}$, then $S(g)=I(g)$. 
\end{proposition}
\begin{proof}
Again, $S(g)$ and $S^*(g)$ coincide, since 
\[
u - K \leq u+ \max\{s, -K\} \leq g
\]
where $u-K$ belongs to the envelope of $S^*(g)$ and $s$ belongs to the envelope of $S(g)$.
\end{proof}
Theorem ~\ref{edwards} will be our main tool in proving the following lemma.
\begin{lemma}\label{uppat-begransad}
Let $\Omega$ be a bounded B-regular domain, and let $\phi$ be a function bounded from above such that $\phi^*=\phi_*$ on $\bar \Omega$. Further assume that $\phi$ has a strong minorant $v$, and that there exists an open set $U_v \subset \Omega$ where $v$ is continuous. Then the envelope
 \[
 P\phi(z) = \sup \big \{ u(z) \suchthat u \in \psh(\Omega), u^* \le \phi \big\}
 \]
is continuous on $U_v$ as well.
\end{lemma}
\begin{proof}
The argument will consist of showing that for $z \in U_v$,
\[
\sup \big\{ u(z) \suchthat u \in \psh(\Omega), u^* \le \phi \big\}=\sup \big\{ u(z) \suchthat u \in \psh(\Omega)\cap C(U_v), u^* \le \phi \big\}.
\]
Since $\phi$ is bounded from above, we may without loss of generality assume that $\phi<0$, and it is enough to consider the set $\psh^-(\Omega)$ of negative plurisubharmonic functions. Define the cones
\begin{s}
\mathcal{F}&=\bigg\{ \Big(\frac{u}{|\phi|}\Big)^* \suchthat u \in \psh^-(\Omega) \cap \usc(\bar \Omega) \bigg \} \\
\mathcal{F}_v&=\bigg\{ \Big(\frac{u}{|\phi|}\Big)^* \suchthat u \in \psh^-(\Omega)\cap \usc(\bar \Omega)\cap C(U_v)\bigg\},
\end{s}
and denote the corresponding operators by $S_z, I_z, S_z^v, I_z^v$, and the corresponding set of Jensen measures by $J_z \subset J_z^v$. Since $\phi$ is continuous on $U_v$, it follows that all elements in $\mathcal{F}_v$ are continuous on $U_v$ as well. As in the proof of Theorem ~\ref{tam-majorant}, we may assume that $v < \phi$, and thus both cones contains the element
\[
\Big(\frac{v}{|\phi|}\Big)^* \leq -\vmathbb{1}.
\]
Furthermore, by Definition ~\ref{majmin}, we necessarily have
\[
\left.\Big(\frac{v}{|\phi|}\Big)^*\right|_{P} = -\infty,
\]
where $P := \{z \suchthat \phi(z)=-\infty\}$. Now fix $z_0 \in U_v$ and pick any measure $\mu \in J^v_{z_0}$. Since 
\[
\int_{P} \Big(\frac{v}{|\phi|}\Big)^*\,d\mu \geq \int_{\bar \Omega} \Big(\frac{v}{|\phi|}\Big)^*\,d\mu \geq \frac{v(z_0)}{|\phi(z_0)|} > -\infty,
\]
we conclude that $\mu(P)=0$. On the other hand, for all $u \in \psh^-(\Omega),$
\[
\Big(\frac{u}{|\phi|}\Big)^* = \frac{u}{|\phi|}
\]
on the set $\bar \Omega \setminus P$. For a specific such $\big(\frac{u}{|\phi|}\big)^* \in \mathcal{F}$, we may by Wikström's approximation theorem \cite{wikstrom} find a sequence $v_k \in \psh^-(\Omega) \cap C(\bar \Omega)$ such that $v_k \searrow u$ on $\bar \Omega$. By previous considerations,
\[
\Big(\frac{v_n}{|\phi|}\Big)^* \searrow \Big(\frac{u}{|\phi|}\Big)^*
\]
on $\Bar \Omega \setminus P$, and thus
\begin{alignat*}{2}
\int_{\bar \Omega} \Big(\frac{u}{|\phi|}\Big)^*\,d\mu &= \int_{\bar \Omega \setminus P} \Big(\frac{u}{|\phi|}\Big)^*\,d\mu &&= \lim_{k\rightarrow \infty} \int_{\bar \Omega \setminus P} \Big(\frac{v_k}{|\phi|}\Big)^*\,d\mu \\
&\geq  \lim_{k\rightarrow \infty} \Big(\frac{v_k(z_0)}{|\phi(z_0)|}\Big)^*  &&= \Big(\frac{u(z_0)}{|\phi(z_0)|}\Big)^*.
\end{alignat*}
We conclude that $J_{z_0} = J_{z_0}^v$. Theorem~\ref{edwards} implies that
\[
S_{z_0}(-\vmathbb{1})=I_{z_0}(-\vmathbb{1})=I^v_{z_0}(-\vmathbb{1})=S^v_{z_0}(-\vmathbb{1}).
\]
The lemma now follows from the fact that
\begin{s}
P\phi(z_0)&=\sup \big\{ u(z_0) \suchthat u \in \psh(\Omega), u^* \le \phi \big\} \\
&= \sup \Big\{ u(z_0) \suchthat u \in \psh^-(\Omega)\cap \usc(\bar \Omega), \Big(\frac{u}{|\phi|}\Big)^* \le -1 \Big\} \\
&= |\phi(z_0)| S_{z_0}(-\vmathbb{1}),
\end{s}
and similarly for the envelope over plurisubharmonic functions continuous on $U_v$.
\end{proof}
\begin{remark}
A similar argument was recently used by Bracci et al. \cite{bracci1} concerning the continuity of the pluricomplex Poisson kernel. 
\end{remark}
It is straightforward to extend this result to the whole set $\{z \in \Omega \suchthat v_*(z) \neq -\infty\}$.
\begin{lemma}\label{lokalkont}
Let $v$ be a nontrivial strong minorant to $\phi$, and let $z_0 \in \Omega$ be such that $v_*(z_0) \neq -\infty$. Then $\phi$ has a nontrivial strong minorant $\tilde v$ which is continuous in a neighborhood of $z_0$.  
\end{lemma}
\begin{proof}
Fix a closed ball $\bar B:= \bar B(z_0, \varepsilon_1) \subset \{z \suchthat v_*(z_0) \neq -\infty\}$. By adding $K|z-z_0|$, where $K>0$ is sufficiently large, we may assume that $v|_{\bd B} > v(z_0)+\varepsilon_2$. Since $v$ is upper semicontinuous at $z_0$, we may further find an open neighborhood $U_{z_0} \subset B$ such that $v|_{U_{z_0}} < v(z_0) + \varepsilon_2$. By the standard gluing lemma,
\[
\tilde v = \begin{cases}
    \max\{v(z) ,v(z_0)+\varepsilon_2\}, & \hfill \text{ if } z\in \bar B,   \\
    v(z),  & \hfill \text{ otherwise,} 
    \end{cases}
\]
is plurisubharmonic, continuous on $U_{z_0}$ and clearly a nontrivial minorant to $\phi$.
\end{proof}
We are now ready to formulate our main theorem.

\begin{theorem}\label{obegransad}
Let $\Omega$ be a bounded B-regular domain, and let $\phi$ be a function having a strong minorant $v$ and a strong majorant $w$ such that $\phi^*=\phi_*$ on $\bar \Omega$. Then the envelope
\[
 P\phi(z) = \sup \big\{u(z) \suchthat u \in \psh(\Omega), u^* \le \phi \big\}
\]
is continuous on $\{z \in \bar\Omega \suchthat v_*(z) \neq -\infty\} \cap \{z \in \bar\Omega \suchthat w^*(z) \neq +\infty\}$.
\end{theorem}
\begin{proof}
Assume without loss of generality that $v \leq P\phi \leq w$, and that the set $\{z \in \Omega \suchthat v_*(z) \neq -\infty\} \cap \{z \in \Omega \suchthat w^*(z) \neq +\infty\}$ contains an element $z_0$. By the proof of Proposition~\ref{nedat-begransad}, $P \phi$ is upper semicontinuous on $\Omega$, and
\[
P\phi(z) = \sup \big\{u(z) \suchthat u \in \psh(\Omega) \cap \usc(\bar \Omega), u \le \phi \big\}
\]
for $z$ such that $w(z) \neq +\infty$. As all functions in this envelope are bounded, Lemma~\ref{uppat-begransad} and Lemma ~\ref{lokalkont} implies that it is enough to take the pointwise supremum over functions continuous on some neighborhood $U_{z_0}$ of $z_0$. In particular, these functions are continuous on $U_{z_0} \cap \{z \in \bar \Omega \suchthat w^*(z) \neq +\infty\}$, and it follows that $P\phi$ is lower semicontinuous at $z_0$.
\end{proof}
\begin{remark}
This theorem is a direct generalization of a classic result in pluripotential theory due to J. B. Walsh \cite{walsh}.
\end{remark}

\end{document}